\documentclass{amsart}%
\usepackage{amssymb}
\usepackage{amsfonts}
\usepackage{amsmath}
\usepackage{graphicx}%
\providecommand{\U}[1]{\protect\rule{.1in}{.1in}}
\newtheorem{theorem}{Theorem}
\theoremstyle{plain}

\newtheorem{corollary}{Corollary}

\newtheorem{definition}{Definition}
\newtheorem{example}{Example}

\newtheorem{lemma}{Lemma}

\newtheorem{proposition}{Proposition}

\numberwithin{equation}{section}
\begin{document}
\title[1-absorbing primary submodules]{1-absorbing primary submodules}
\author{Ece Yetkin Celikel}
\address{Department of Electrical-Electronics Engineering, Faculty of Engineering,
Hasan Kalyoncu University, Gaziantep, Turkey}
\email{ece.celikel@hku.edu.tr, yetkinece@gmail.com}
\thanks{This paper is in final form and no version of it will be submitted for
publication elsewhere.}
\date{January, 2021}

\begin{abstract}
Let $R$ be a commutative ring with non-zero identity and $M$ be a unitary
$R$-module. The goal of this paper is to extend the concept of 1-absorbing
primary ideals to 1-absorbing primary submodules. A proper submodule $N$ of
$M$ is said to be a 1-absorbing primary submodule if whenever non-unit
elements $a,b\in R$ and $m\in M$ with $abm\in N$, then either $ab\in(N:_{R}M)$
or $m\in M-rad(N).$ Various properties and chacterizations of this class of
submodules are considered. Moreover, 1-absorbing primary avoidance theorem is proved.

\end{abstract}
\maketitle

\section{Introduction}

Throughout this paper, we shall assume unless otherwise stated, that all rings
are commutative with non-zero identity and all modules are considered to be
unitary. A \textit{prime} (resp. \textit{primary}) submodule is a proper
submodule $N$ of $M$ with the property that for $a\in R$ and $m\in M$, $am\in
N$ implies that $m\in N$ or $a\in(N:_{R}M)$ (resp. $a\in\sqrt{(N:_{R}M)}$).
Since prime and primary ideals (submodules) have an important role in the
theory of modules over commutative rings, generalizations of these concepts
have been studied by several authors \cite{Ame}-\cite{YF1}, \cite{M},
\cite{pay}. For a survey article consisting some of generalizations see
\cite{Badlast}. In 2007, Badawi \cite{B} called a non-zero proper ideal $I$ of
$R$ a \textit{2-absorbing ideal} of $R$ if whenever $a,b,c\in R$ and $abc\in
I$, then $ab\in I$ or $ac\in I$ or $bc\in I$. As an extension of 2-absorbing
primary ideals, the concept of 2-absorbing submodules are introduced by Darani
and Soheilnia \cite{YF1} and studied by Payrovi, Babaei \cite{pay}. We recall
from \cite{YF1} that a proper submodule $N$ of $M$ is said to be a
\textit{2-absorbing submodule} if whenever $a,b\in R$ and $m\in M$ with
$abm\in N$, then $ab\in(N:_{R}M)$ or $am\in N$ or $bm\in N$. In 2014, Badawi,
Tekir and Yetkin \cite{Badawi2} introduced the concept of $2$-absorbing
primary ideals. A proper ideal $I$ of $R$ is called \textit{2-absorbing
primary} if whenever $a,b,c\in R$ with $abc\in I$, then $ab\in I$ or
$ac\in\sqrt{I}$ or $bc\in\sqrt{I}$. After that, the notion of 2-absorbing
pimary submodules is introduced and studied in \cite{M}. According to
\cite{M}, a proper submodule $N$ of $M$ is said to be \textit{2-absorbing
primary }provided that $a,b\in R,$ $m\in M$ and $abm\in N$ imply either
$ab\in(N:_{R}M)$ or $am\in M$-$rad(N)$ or $bm\in M$-$rad(N).$ As a recent
study, the class of \textit{1-absorbing primary ideals }was defined in
\cite{badece}. According to \cite{badece}, a proper ideal $I$ of $R$ is said
to be a \textit{1-absorbing primary ideal} if whenever non-unit elements
$a,b,c$ of $R$ and $abc\in I$, then $ab\in I$ or $c\in\sqrt{I}.$ Our aim is to
extend the notion of 1-absorbing primary ideals to 1-absorbing primary submodules.

For the sake of thoroughness, we give some definitions which we will need
throughout this study. Let $I$ be an ideal of a ring $R$. By $\sqrt{I},$ we
mean the radical of $I$ which is the intersection of all prime ideals
containing $I$, that is $\{r\in R:r^{n}\in I$ for some $n\}$. Let $M$ be an
$R$-module and $N$ be a submodule of $M$. We will denote by $(N:_{R}M)$
\textit{the residual of $N$ by} $M$, that is, the set of all $r\in R$ such
that $rM\subseteq N$.\ The annihilator of $M$ denoted by $Ann_{R}(M)$ is
$(0:_{R}M)$. The $M$-radical of $N$, denoted by $M$-$\mathrm{rad}(N)$, is
defined to be the intersection of all prime submodules of $M$ containing $N$.
If $M$ is a multiplication $R$-module, then $M-rad(N)=\{m\in M:m^{k}\subseteq
N$ for some $k\geq0\}$ \cite[Theorem 3.13]{Ame}. If there is no such a prime
submodule, then $M-rad(N)=M.$ For the other notations and terminologies that
are used in this article, the reader is referred to \cite{Huc}.

We summarize the content of this article as follows. We call a proper
submodule $N$ of $M$ a 1-absorbing primary submodule if whenever non-unit
elements $a,b\in R$ and $m\in M$ with $abm\in N$, then $ab\in(N:_{R}M)$ or
$m\in M-rad(N).$ It is clear that a prime submodule is a 1-absorbing primary
submodule, and a 1-absorbing primary submodule is a 2-absorbing primary
submodule. In Section 2, we start with examples (Example \ref{ex1} and Example
\ref{ex2}) showing that the inverses of these implications are not true in
general. Various characterizations for 1-absorbing primary submodules are
given (Theorem \ref{l1}, Theorem \ref{tn} and Theorem \ref{t0}). Moreover, the
behavior of 1-absorbing primary submodules in modules under homomorphism,
module localizations and direct product of modules are investigated
(Proposition \ref{f}, Proposition \ref{tc} and Proposition \ref{S}). Finally,
in Section 3, the 1-absorbing primary avoidance theorem is proved.

\section{Properties of 1-absorbing primary submodules}

\begin{definition}
Let $M$ be a module over a commutative ring $R$ and $N$ be a proper submodule
of $M$. We call $N$ a 1-absorbing primary submodule if whenever non-unit
elements $a,b\in R$ and $m\in M$ with $abm\in N$, then $ab\in(N:_{R}M)$ or
$m\in M-rad(N).$
\end{definition}

It is clear that the following implication hold: prime submodule $\Rightarrow$
1-absorbing primary submodule $\Rightarrow$ 2-absorbing primary submodule. The
following example shows that a 1-absorbing primary submodule of $M$ needs not
to be a primary (prime) submodule; and also there are 2-absorbing primary
submodules which are not 1-absorbing primary.

\begin{example}
\begin{enumerate}
\item \label{ex1}Let $A=K[x,y]$, where $K$ is a field, $Q=(x,y)A$. Consider
$R=A_{Q}$ and $M=R$ as an $R$-module. Then $N=(x^{2},xy)M$ is a 1-absorbing
primary submodule of $M$ \cite[Example 1]{badece}. Observe that $\sqrt
{(N:_{R}M)}=xR$. Since $x\cdot y\in N$, but $x\notin N$ and $y\notin
\sqrt{(N:_{R}M)},$ $N$ is not a primary submodule (so, it is not a prime
submodule) of $M$.

\item Consider the submodule $N=p^{2}q%
\mathbb{Z}
$ of $%
\mathbb{Z}
$-module $%
\mathbb{Z}
$ where $p$ and $q$ are distinct prime integers. Then $N$ is a 2-absorbing
primary submodule of $%
\mathbb{Z}
$ by \cite[Corollary 2.21]{M}. However it is not 1-absorbing primary as
$p\cdot p\cdot q\in N$ but neither $p\cdot p\in(N:_{%
\mathbb{Z}
}%
\mathbb{Z}
)=N$ nor $q\in M-rad(N)=pq%
\mathbb{Z}
$.
\end{enumerate}
\end{example}

The following example shows that there are some modules of which every proper
submodule is 2-absorbing primary but which has no 1-absorbing primary submodule.

\begin{example}
\label{ex2}Let $p$ be a fixed prime integer. Then $E(p)=\{\alpha\in%
\mathbb{Q}
/%
\mathbb{Z}
:\alpha=r/p^{n}+%
\mathbb{Z}
$ for some $r\in%
\mathbb{Z}
$ and $n\in%
\mathbb{N}
\cup(0)\}$ is a non-zero submodule of $%
\mathbb{Z}
$-module $%
\mathbb{Q}
/%
\mathbb{Z}
$. For each $t\in%
\mathbb{N}
$, set $G_{t}=\{\alpha\in%
\mathbb{Q}
/%
\mathbb{Z}
:\alpha=r/p^{t}+Z$ for some $r\in%
\mathbb{Z}
\}$ . Observe that each proper submodule of $E(p)$ is equal to $G_{i}$ for
some $i\in%
\mathbb{N}
$ and $(G_{t}:_{%
\mathbb{Z}
}E(p))=0$ for every $t\in%
\mathbb{N}
.$ It is shown in \cite[Example 1]{At} that every submodule $G_{t}$ is not a
primary submodule of $E(p).$ Thus there is no prime submodule in $E(p).$ Thus
$E(p)-rad(G_{t})=E(p)$. Therefore, each $G_{t}$ is a 1-absorbing primary
submodule of $R$.
\end{example}

We next give several characterizations of 1-absorbing primary submodules of an
$R$-module.

\begin{theorem}
\label{l1}Let $N$ be a proper submodule of an $R$-module $M$. Then the
following statements are equivalent:
\end{theorem}

\begin{enumerate}
\item $N$ is a 1-absorbing primary submodule of $M.$

\item If $a,b$ are non-unit elements of $R$ such that $ab\notin(N:_{R}M),$
then $(N:_{M}ab)\subseteq M-rad(N).$

\item If $a,b$ are non-unit elements of $R$, and $K$ is a submodule of $M$
with $abK\subseteq N$, then $ab\in(N:_{R}M)$ or $K\subseteq M-rad(N).$

\item If $I_{1}I_{2}K\subseteq N$ for some proper ideals $I_{1},I_{2}$ of $R$
and some submodule $K$ of $M$, then either $I_{1}I_{2}\subseteq(N:_{R}M)$ or
$K\subseteq M-rad(N).$
\end{enumerate}

\begin{proof}
(1)$\Rightarrow$(2) Suppose that $a,b$ are non-unit elements of $R$ such that
$ab\notin(N:_{R}M).$ Let $m\in(N:_{M}ab)$. Hence $abm\in N$. Since $N$ is
1-absorbing primary submodule and $ab\notin(N:_{R}M),$ we have $m\in
M-rad(N),$ and so $(N:_{M}ab)\subseteq M-rad(N).$

(2)$\Rightarrow$(3) Suppose that $ab\notin(N:_{R}M)$. Since $abK\subseteq N,$
we have $K\subseteq(N:_{M}ab)\subseteq M-rad(N)$\ by (2).

(3)$\Rightarrow$(4) Assume on the contrary that neither $I_{1}I_{2}%
\subseteq(N:_{R}M)$ nor $K\subseteq M-rad(N).$ Then there exist non-unit
elements $a\in I_{1},$ $b\in I_{2}$ with $ab\notin(N:_{R}M)$. Thus
$abK\subseteq N$, it contradicts with (3).

(4)$\Rightarrow$(1) Let $a,b\in R$ be non-unit elements, $m\in M$ and $abm\in
N$. Put $I_{1}=aR$, $I_{2}=bR$, $K=Rm.$ Thus the result is clear.
\end{proof}

An $R$-module $M$ is called a \textit{multiplication module} if every
submodule $N$ of $M$ has the form $IM$ for some ideal $I$ of $R$.
Equivalently, $N=(N:_{R}M)M$ \cite{Smith}. Let $M$ be a multiplication
$R$-module and let $N=IM$ and $K=JM$ for some ideals $I$ and $J$ of $R$. The
product of $N$ and $K$ is denoted by $NK$ is defined by $IJM$. Clearly, $NK$
is a submodule of $M$ and contained in $N\cap K$. It is shown in \cite[Theorem
3.4]{Ame} that the product of $N$ and $K$ is independent of presentations of
$N$ and $K$. It is shown in \cite[Theorem 2.12]{Smith} that if $N$ is a proper
submodule of a multiplication $R$-module $M$, then $M$-$\mathrm{rad}%
(N)=\sqrt{(N:_{R}M)}M$. If $M$ is a finitely generated multiplication
$R$-module, then $(M$-$\mathrm{rad}(N):M)=\sqrt{(N:_{R}M)}$ by \cite[Lemma
2.4]{M}. Now, we are ready for characterizing 1-absorbing primary submodules
of finitely generated multiplication module $M$ in terms of submodules of $M$.

\begin{theorem}
\label{tn}Let $M$ be a finitely generated multiplication $R$-module and $N$ be
a proper submodule of $M.$ Then the following statements are equivalent:
\end{theorem}

\begin{enumerate}
\item $N$ is a 1-absorbing primary submodule of $M$.

\item If $N_{1}N_{2}N_{3}\subseteq N$ for some submodules $N_{1},N_{2}$%
,$N_{3}$ of $M,$ then either $N_{1}N_{2}\subseteq N$ or $N_{3}\subseteq
M-rad(N).$
\end{enumerate}

\begin{proof}
(1)$\Rightarrow$(2) Suppose that $N$ is a 1-absorbing primary submodule of
$M$, $N_{1}N_{2}N_{3}\subseteq N$ and $N_{3}\nsubseteq M-rad(N).$ Since $M$ is
a finitely generated multiplication module, $N_{1}=I_{1}M$ and $N_{2}=I_{2}M$
for some ideals $I_{1},I_{2}$ of $R$. Hence $I_{1}I_{2}N_{3}\subseteq N$.
Since $N_{3}\nsubseteq M-rad(N),$ we have $I_{1}I_{2}\subseteq(N:_{R}M)$ by
Theorem \ref{l1}. Thus we conclude $N_{1}N_{2}\subseteq(N:M)M=N$.

(2)$\Rightarrow$(1) Let $I_{1}I_{2}K\subseteq N.$ Then there exists an ideal
$I_{3}$ of $R$ such that $I_{1}I_{2}I_{3}M\subseteq N$ which gives $I_{1}%
I_{2}M\subseteq N$ or $I_{3}M\subseteq M-rad(N).$ By \cite[p.231
Corollary]{Smith2}, we have $I_{1}I_{2}\subseteq(N:_{R}M)+Ann_{R}%
(M)=(N:_{R}M)$ or $K\subseteq M-rad(N).$ Hence, we are done from Theorem
\ref{l1}.
\end{proof}

\begin{lemma}
\cite[Theorem 10]{Smith2}\label{9}Let $M$ be a finitely generated faithful
multiplication $R$-module, then $(IM:M)=I$ .for all ideals $I$ of $R$.
\end{lemma}

In \cite[Corollary 2]{At}, for a proper submodule $N$ of a multiplication
$R$-module $M,$ it is shown that $N$ is primary submodule of $M$ if and only
if $(N:_{R}M)$ is primary ideal of $R$. Analogous with this result, we have
the following.

\begin{theorem}
\label{t0}Let $I$ be an ideal of a ring $R$ and $N$ be a submodule of a
finitely generated faithful multiplication $R$-module $M$. Then
\end{theorem}

\begin{enumerate}
\item $I$ is a 1-absorbing primary ideal of $R$ if and only if $IM$ is a
1-absorbing primary submodule of $M$.

\item $N$ is a 1-absorbing primary submodule of $M$ if and only if $(N:M)$ is
a 1-absorbing primary ideal of $R$.

\item $N$ is a 1-absorbing primary submodule of $M$ if and only if $N=IM$ for
some 1-absorbing primary ideal of $R$.
\end{enumerate}

\begin{proof}
(1) Suppose $I$ is a 1-absorbing primary ideal of $R$. If $IM=M$, then
$I=(IM:M)=R$ by Lemma \ref{9}, a contradiction. Thus, $IM$ is proper in $M$.
Now, let $a,b\in R$ be non-unit elements and $m\in M$ such that $abm\in IM$
and $ab\notin(IM:M)=I$. Then $ab((m):M)=((abm):M)\subseteq(IM:M)\subseteq
(\sqrt{I}M:M)=\sqrt{I}$. Since $I$ is a 1-absorbing primary ideal, we conclude
that $((m):M)\subseteq\sqrt{I}$. Thus, $m\in((m):M)M\subseteq\sqrt{I}%
M=M$-$rad(IM)$. Conversely, suppose $IM$ is 1-absorbing primary submodule of
$M$. Then clearly $I$ is proper in $R$. Let $a,b,c\in R$ be non-unit elements
with $abc\in I$ and $ab\notin(IM:M)=I$. Since $abM\in IM$ and $IM$ is a
1-absorbing primary submodule, then $cM\subseteq M$-$rad(IM)=\sqrt{I}M$.
Therefore, $c\in(\sqrt{I}M:M)=\sqrt{I}$ and $I$ is a 1-absorbing primary ideal
of $R$.

(2) Since $N=(N:M)M$, it follows by (1).

(3) Putting $I=(N:M)$ in (2), the claim is clear.
\end{proof}

The following example shows that if $(N:_{R}M)$ is a 1-absorbing primary ideal
of $R,$ then $N$ is not needed to be a 1-absorbing primary submodule in general.

\begin{example}
Let $R=%
\mathbb{Z}
$ and $M=%
\mathbb{Z}
\times%
\mathbb{Z}
$ be an $R$-module and $p$ a prime integer. Consider the submodule $N=p^{n}%
\mathbb{Z}
\times\{0\}$ of $M$ for $n\geq2.$ Then $(N:_{R}M)=\{0\}$ is a 1-absorbing
primary ideal of $R$. However, $N$ is not a 1-absorbing primary submodule of
$M$ since $p\cdot p^{n-1}\cdot(1,0)\in N$ but neither $p\cdot p^{n-1}=p^{n}%
\in(N:_{R}M)=\{0\}$ nor $(1,0)\in M-rad(N)=p%
\mathbb{Z}
\times\{0\}$.
\end{example}

In view of Theorem \ref{t0}, we conclude the following result.

\begin{proposition}
\label{t1}Let $M$ be a finitely generated multiplication $R$-module and $N$ be
a 1-absorbing primary submodule of $M$. Then the following are satisfied:
\end{proposition}

\begin{enumerate}
\item $\sqrt{(N:_{R}M)}$ is a prime ideal of $R$.

\item $\sqrt{(N:_{R}m)}$ is a prime ideal of $R$ containing $\sqrt{(N:_{R}%
M)}=P$ for every $m\notin M-rad(N)$.

\item $M-rad(N)$ is a prime submodule of $M$.
\end{enumerate}

\begin{proof}
(1) Let $N$ be a 1-absorbing primary submodule of $M$. Then $(N:_{R}M)$ is
1-absorbing primary\ ideal of $R$ by Theorem \ref{t0}. From \cite[Theorem
2]{badece}, we conclude that$\sqrt{(N:_{R}M)}$ is a prime ideal of $R$.

(2) Since $N$ is a 1-absorbing primary ideal, $\sqrt{(N:_{R}M)}=P$ is a prime
ideal of $R$ by (1). Suppose that $a,b\in R$ such that $ab\in\sqrt{(N:_{R}m)}%
$. Without loss of generality we may assume that $a$ and $b$ are non-unit
elements of $R$. Then there exists a positive integer $n$ such that
$a^{n}b^{n}m\in N$. Since $N$ is 1-absorbing primary submodule, and $m\notin
M-rad(N)$, it implies that either $(ab)^{n}\in(N:_{R}M)$. Since $P$ is prime
and $ab\in P$, we conclude either $a\in P=\sqrt{(N:_{R}M)}\subseteq
\sqrt{(N:_{R}m)}$ or $b\in P=\sqrt{(N:_{R}M)}\subseteq\sqrt{(N:_{R}m)}.$

(3) Suppose thet $N$ is a 1-absorbing primary submodule. Since $\sqrt
{(N:_{R}M)}$ is a prime ideal of $R$ by (1), we conclude that $M-rad(N)=\sqrt
{(N:_{R}M)}M$ is a prime submodule of $M$ by \cite[Corollary 2.11]{Smith}.
\end{proof}

Note that the intersection of two distinct non-zero 1-absorbing primary
submodules need not be a 1-absorbing primary submodule. Consider $%
\mathbb{Z}
$-module $%
\mathbb{Z}
$. Then $2%
\mathbb{Z}
$ and $3%
\mathbb{Z}
$ are clearly 1-absorbing primary submodules but $2%
\mathbb{Z}
\cap3%
\mathbb{Z}
=6%
\mathbb{Z}
$ is not. Indeed, $2\cdot2\cdot3\in6%
\mathbb{Z}
$ but neither $2\cdot2\in(%
\mathbb{Z}
:6%
\mathbb{Z}
)=6%
\mathbb{Z}
$ nor $3\in%
\mathbb{Z}
-rad(6%
\mathbb{Z}
)=6%
\mathbb{Z}
.$ We call a proper submodule $N$ of $M$ a $P$-1-absorbing submodule of $M$ if
$\sqrt{(N:_{R}M)}=P$ is a prime submodule of $R.$ In the next theorem, we show
that if $N_{i}$'s are $P$-1-absorbing primary submodules of a multiplication
module $M$, then the intersection of these submodules is a $P$-1-absorbing
primary submodule of $M$.

\begin{proposition}
\label{int}Let $M$ be a multiplication $R$-module. If $\{N_{i}\}_{i=1}^{k}$ is
a family of $P$-1-absorbing primary submodules of $M$, then so is $%
{\displaystyle\bigcap\limits_{i=1}^{k}}
N_{i}$.
\end{proposition}

\begin{proof}
Suppose that $abm\in%
{\displaystyle\bigcap\limits_{i=1}^{k}}
N_{i}$ but $ab\notin\left(
{\displaystyle\bigcap\limits_{i=1}^{k}}
N_{i}:_{R}M\right)  $ for non-unit elements $a,b\in R$ and $m\in M$. Then
$ab\notin(N_{j}:_{R}M)$ for some $j\in\{1,...,k\}.$ Since $N_{j}$ is
1-absorbing primary and $abm\in N_{j}$, we have $m\in M-rad(N_{j})$. Now,
since $M-rad\left(
{\displaystyle\bigcap\limits_{i=1}^{k}}
N_{i}\right)  =%
{\displaystyle\bigcap\limits_{i=1}^{k}}
M-rad(N_{i})=PM$ by \cite[Proposition 2.14 (3)]{M}, we are done.
\end{proof}

\begin{lemma}
\cite{Lu2}\label{l4}Let $\varphi:M_{1}\longrightarrow M_{2}$ be an $R$-module
epimorphism. Then

\begin{enumerate}
\item If $N$ is a submodule of $M_{1}$ and $\ker\left(  \varphi\right)
\subseteq N$, then $\varphi\left(  M_{1}\text{-}rad(N)\right)  =M_{2}%
$-$rad(\varphi\left(  N\right)  )$.

\item If $K$ is a submodule of $M_{2}$, then $\varphi^{-1}\left(
M_{2}\text{-}rad(K)\right)  =M_{1}$-$rad(\varphi^{-1}\left(  K\right)  )$.
\end{enumerate}
\end{lemma}

\begin{proposition}
\label{f}Let $M_{1}$ and $M_{2}$ be $R$-modules and $f:M_{1}\rightarrow M_{2}$
be a module homomorphism. Then the following statements hold:
\end{proposition}

\begin{enumerate}
\item If $N_{2}$ is a 1-absorbing primary submodule of $M_{2}$, then
$f^{-1}(N_{2})$ is a 1-absorbing primary submodule of $M_{1}$.

\item Let $f$ be an epimorphism. If $N_{1}$ is a 1-absorbing primary submodule
of $M_{1}$ containing $Ker(f)$, then $f(N_{1})$ is a 1-absorbing primary
submodule of $M_{2}.${}
\end{enumerate}

\begin{proof}
(1) Suppose that $a,b$ are non-unit elements of $R$, $m_{1}\in M_{1}$ and
$abm_{1}\in f^{-1}(N_{2})$. Then $abf(m_{1})\in N_{2}.$ Since $N_{2}$ is
1-absorbing primary, we have either $ab\in(N_{2}:_{R}M_{2})$ or $f(m_{1})\in
M_{2}-rad(N_{2}).$ Here, we show that $(N_{2}:_{R}M_{2})\subseteq(f^{-1}%
(N_{2}):_{R}M_{1})$. Let $r\in(N_{2}:_{R}M_{2})$. Then $rM_{2}\subseteq N_{2}$
which implies that $rf^{-1}(M_{2})\subseteq f^{-1}(N_{2})$; i.e.
$rM_{1}\subseteq f^{-1}(N_{2}).$ Thus $r\in(f^{-1}(N_{2}):_{R}M_{1}).$ Hence
$ab\in(f^{-1}(N_{2}):_{R}M_{1})$ or $m_{1}\in f^{-1}(M_{2}-rad(N_{2})).$ Since
$f^{-1}(M_{2}-rad(N_{2}))=M_{1}-rad(f^{-1}(N_{2}))$ by Lemma \ref{l4} (2) and
$f^{-1}(N_{2})$ is a 1-absorbing primary submodule of $M_{1}$.

(2) Suppose that $a$,$b$ are non-unit elements of $R$, $m_{2}\in M_{2}$ and
$abm_{2}\in f(N_{1})$. Since $f$ is an epimorphism, there exists $m_{1}\in
M_{1}$ such that $f(m_{1})=m_{2}.$ Since $Kerf\subseteq N_{1},$ $abm_{1}\in
N_{1}.$ Hence $ab\in(N_{1}:_{R}M_{1})$ or $m_{1}\in M_{1}-rad(N_{1})$. Here,
we show that $(N_{1}:_{R}M_{1})\subseteq(f(N_{1}):_{R}M_{2})$. Let $r\in
(N_{1}:_{R}M_{1})$. Then $rM_{1}\subseteq N_{1}$ which implies that
$rf(M_{1})\subseteq f(N_{1})$. Since $f$ is onto, we conclude that
$rM_{2}\subseteq f(N_{1}),$ that is, $r\in(f(N_{1}):_{R}M_{2}).$ Thus
$ab\in(f(N_{1}):_{R}M_{2})$ or $m_{2}=f(m_{1})\in f(M_{1}-rad(N_{1}%
))=M_{2}-rad(f(N_{1}))$ by Lemma \ref{l4} (1)$,$ as desired.
\end{proof}

As a consequence of Theorem \ref{f}, we have the following result.

\begin{corollary}
\label{c/}Let $M$ be an $R$-module and $N_{1},$ $N_{2}$ be submodules of $M$
with $N_{2}\subseteq N_{1}$. Then $N_{1}$ is a 1-absorbing primary submodule
of $M$ if and only if $N_{1}/N_{2}$ is a 1-absorbing primary submodule of
$M/N_{2}$.
\end{corollary}

\begin{proof}
Suppose that $N_{1}$ is a 1-absorbing primary submodule of $M$. Consider the
canonical epimorphism $f:M\rightarrow M/N_{2}$ in Proposition \ref{f}. Then
$N_{1}/N_{2}$ is a 1-absorbing primary submodule of $M/N_{2}$. Conversely, let
$a$ and $b$ are non-unit elements of $R$, $m\in M$ such that $abm\in N_{1}.$
Hence $ab(m+N_{2})\in N_{1}/N_{2}$. Since $N_{1}/N_{2}$ is a 1-absorbing
primary submodule of $M/N_{2}$, it implies either $ab\in(N_{1}/N_{2}%
:_{R}M/N_{2})$ or $m+N_{2}\in M/N_{2}-rad(N_{1}/N_{2})=M-rad(N_{1})/N_{2}.$
Therefore $ab\in(N_{1}:_{R}M)$ or $m\in M-rad(N_{1}).$ Thus $N_{1}$ is a
1-absorbing primary submodule of $M$.
\end{proof}

Let $M_{1}$ be $R_{1}$-module and $M_{2}$ be $R_{2}$-module where $R_{1}$ and
$R_{2}$ are commutative rings with identity. Let $R=R_{1}\times R_{2}$ and
$M=M_{1}\times M_{2}.$ Then $M$ is an $R$-module and every submodule of $M$ is
of the form $N=N_{1}\times N_{2}$ for some submodules $N_{1},N_{2}$ of $M_{1}%
$, $M_{2}$, respectively. Also, $M-rad(N_{1}\times N_{2})=M_{1}-rad(N_{1}%
)\times M_{2}-rad(N_{2})$ by \cite[Lemma 2.3 (ii)]{Alk}$.$

\begin{proposition}
\label{tc}Let $M_{1}$ be an $R_{1}$-module and $M_{2}$ be an $R_{2}$-module,
where $R_{1}$, $R_{2}$ are commutative rings with identity, $R=R_{1}\times
R_{2}$ and $M=M_{1}\times M_{2}$. Suppose that $N_{1}$ is a proper submodule
of $M_{1}$. If $N=N_{1}\times M_{2}$ is a 1-absorbing primary submodule of
$R$-module $M$, then $N_{1}$ is a 1-absorbing primary submodule of $R_{1}%
$-module $M_{1}.$
\end{proposition}

\begin{proof}
Suppose that $N=N_{1}\times M_{2}$ is a 1-absorbing primary submodule of $M$.
Put $M^{\prime}=M/\{0\}\times M_{2}$ and $N^{\prime}=N/\{0\}\times N_{2}.$
From Corollary \ref{c/}, $N^{\prime}$ is a 1-absorbing primary submodule of
$M^{\prime}.$ Since $M^{\prime}\cong M_{1}$ and $N^{\prime}\cong N_{1}$, we
conclude the result.
\end{proof}

\begin{proposition}
\label{S}Let $S$ be a multiplicatively closed subset of a commutative ring $R$
and $M$ be an $R$-module. If $N$ is a 1-absorbing primary submodule of $M$ and
$S^{-1}N\neq S^{-1}M,$ then $S^{-1}N$ is a 1-absorbing primary submodule of
$S^{-1}R$-module $S^{-1}M.$
\end{proposition}

\begin{proof}
Let $\frac{a}{s_{1}}$ and $\frac{b}{s_{2}}$ be non-unit elements of $S^{-1}R,$
$\frac{m}{s_{3}}\in S^{-1}M$ with $\frac{a}{s_{1}}\frac{b}{s_{2}}\frac
{m}{s_{3}}\in S^{-1}N$. Hence $tabm\in N$ for some $t\in S$. Since $N$ is
1-absorbing primary, we have either $tm\in M-rad(N)$ or $ab\in(N:_{R}M).$ Thus
we conclude either $\frac{m}{s_{3}}=\frac{tm}{ts_{3}}\in S^{-1}%
(M-rad(N))\subseteq S^{-1}M-rad(S^{-1}N)$ or $\frac{ab}{s_{1}s_{2}}\in
S^{-1}(N:_{R}M)\subseteq(S^{-1}N:_{S^{-1}R}S^{-1}M)$.
\end{proof}

Let $R$ be a ring and $M$ be an $R$-module. The idealization of $M$ is denoted
by $R(M)=R(+)M$ is a commutative ring with identity with coordinate-wise
addition and multiplication defined by $(a,m_{1})(b,m_{2})=(ab,$
$am_{2}+bm_{1}).$ An ideal $H$ is called homogeneous if $H=I(+)N$ for some
ideal $I$ of $R$ and some submodule $N$ of $M$ such that $IM\subseteq N$.

\begin{proposition}
\label{id}Let $M$ be an $R$-module and $I(+)N$ be a homogeneous ideal of
$R(M).$ If $I(+)N$ is a 1-absorbing primary ideal of $R(M)$, then $I$ is a
1-absorbing primary ideal of $R.$
\end{proposition}

\begin{proof}
Suppose that $a,b,c$ are non-unit elements of $R$ such that $abc\in I$ and
$c\notin\sqrt{I}$. Then $(a,0_{M})\cdot(b,0_{M})\cdot(c,0_{M})\in I(+)N.$ Note
that $\sqrt{I(+)N}=\sqrt{I}(+)M$ by \cite[Theorem 25.1 (5)]{Huc}. Then
$(c,0_{M})\notin\sqrt{I(+)N}$. Since $I(+)N$ is 1-absorbing primary, we
conclude that $(a,0_{M})\cdot(b,0_{M})\in I(+)N.$ Thus $ab\in I$, we are done.
\end{proof}

\section{The 1-absorbing primary avoidance theorem}

In this section, we prove the 1-absorbing primary avoidance theorem.
Throughout this section, let $M$ be a finitely generated multiplication
$R$-module and $N,$ $N_{1},...,N_{n}$ be submodules of $M$. Recall from
\cite{Lu} that a covering $N\subseteq N_{1}\cup N_{2}\cup\cdots\cup N_{n}$ is
said to be efficient if no $N_{k}$ is superfluous. Also, $N=N_{1}\cup
N_{2}\cup\cdots\cup N_{n}$ is an efficient union if none of the $N_{k}$ may be
excluded. A covering of a submodule by two submodules is never efficient.

\begin{theorem}
\label{ef}Let $N\subseteq N_{1}\cup N_{2}\cup\cdots\cup N_{n}$ be an efficient
covering of submodules $N_{1},N_{2},...,N_{n}$ of $M$ where $n>2.$ If
$\sqrt{(N_{i}:_{R}M)}\nsubseteq\sqrt{(N_{j}:_{R}m)}$ for all $m\in M\backslash
M-rad(N_{j})$ whenever $i\neq j$, then no $N_{i}$ ($1\leq i\leq n$) is a
1-absorbing primary submodule of $M$.
\end{theorem}

\begin{proof}
Assume on the contrary that $N_{k}$ is a 1-absorbing primary submodule of $M$
for some $1\leq k\leq n$. Since $N\subseteq\cup N_{i}$ is an efficient
covering, $N\subseteq%
{\displaystyle\bigcup}
(N_{i}\cap N)$ is also an efficient covering. From \cite[Lemma 2.1]{Lu},
$\left(
{\displaystyle\bigcap\limits_{i\neq k}}
N_{i}\right)  \cap N\subseteq N_{k}\cap N$. Here, observe that $\sqrt
{(N_{i}:_{R}M)}$ is a proper ideal of $R$ for all $1\leq i\leq n$. Also, from
our assumption, there is a non-unit element $a_{i}\in\sqrt{(N_{i}:_{R}%
M)}\backslash\sqrt{(N_{k}:_{R}m)}$ for all $i\neq k$ and for all $m\in
M\backslash M-rad(N_{k})$. Then there is a positive integer $n_{i}$ such that
$a_{i}^{n_{i}}\in(N_{i}:_{R}M)$ for each $i\neq k$. Put $a=%
{\displaystyle\prod\limits_{i=1}^{k-1}}
a_{i}$, $b=%
{\displaystyle\prod\limits_{i=k+1}^{n}}
a_{i}$ and $n=\max\{n_{1},...,n_{k-1},n_{k+1},...,n_{n}\}$. Now, we show that
$a^{n}b^{n}m\in\left(  \left(
{\displaystyle\bigcap\limits_{i\neq k}}
N_{i}\right)  \cap N\right)  \backslash\left(  N_{k}\cap N\right)  .$ Suppose
that $a^{n}b^{n}m\in N_{k}\cap N$. Then $a^{n}b^{n}\in(N_{k}:_{R}%
m)\subseteq\sqrt{(N_{k}:_{R}m)}.$ By Theorem \ref{t1} (2), $\sqrt{(N_{k}%
:_{R}m)}$ is a prime ideal. It implies that $a\in\sqrt{(N_{k}:_{R}m)}$ or
$b\in\sqrt{(N_{k}:_{R}m)}$. Thus $a_{i}\in\sqrt{(N_{k}:_{R}m)}$ for some
$i\neq k$ , a contradiction. Therefore $a^{n}b^{n}m\in\left(  \left(
{\displaystyle\bigcap\limits_{i\neq k}}
N_{i}\right)  \cap N\right)  \backslash\left(  N_{k}\cap N\right)  $ which is
a contradiction. Thus $N_{k}$ is not a 1-absorbing primary submodule.
\end{proof}

\begin{theorem}
\label{av}(1-absorbing Primary Avoidance Theorem) Let $N,$ $N_{1},$
$N_{2},...,N_{n}$ $(n\geq2)$ be submodules of $M$ such that at most two of
$N_{1},N_{2},...,N_{n}$ are not 1-absorbing primary with $N\subseteq N_{1}\cup
N_{2}\cup\cdots\cup N_{n}.$ If $\sqrt{(N_{i}:_{R}M)}\nsubseteq\sqrt
{(N_{j}:_{R}m)}$ for all $m\in M\backslash M-rad(N_{j})$ whenever $i\neq j$,
then $N\subseteq N_{k}$ for some $1\leq k\leq n.$
\end{theorem}

\begin{proof}
Since it is clear for $n\leq2,$ suppose that $n>2.$ Since any cover consisting
submodules of $M$ can be reduced to an efficient one by deleting any
unnecessary terms, we may assume that $N\subseteq N_{1}\cup N_{2}\cup
\cdots\cup N_{n}$ is an efficient covering of submodules of $M.$ From Theorem
\ref{ef}, it implies that no $N_{k}$ is a 1-absorbing primary submodule which
contradicts with the hypothesis. Thus $N\subseteq N_{k}$ for some $1\leq k\leq
n.$
\end{proof}

\begin{corollary}
Let $N$ be a proper submodule of $M$. If 1-absorbing primary avoidance theorem
holds for $M$, then the 1-absorbing primary avoidance theorem holds for $M/N.$
\end{corollary}

\begin{proof}
Let $K/N,$ $N_{1}/N,$ $N_{2}/N,...,N_{n}/N$ $(n\geq2)$ be submodules of $M/N$
such that at most two of $N_{1}/N,N_{2}/N,...,N_{n}/N$ are not 1-absorbing
primary and $K/N\subseteq N_{1}/N\cup N_{2}/N\cup\cdots\cup N_{n}/N.$ Hence,
$K\subseteq N_{1}\cup N_{2}\cup\cdots\cup N_{n}$ and at most two of $N_{1},$
$N_{2},...,$ $N_{n}$ are not 1-absorbing primary by Corollary \ref{c/}.
Suppose that \linebreak$\sqrt{(N_{i}/N:_{R}M/N)}\nsubseteq\sqrt{(N_{j}%
/N:_{R}m+N)}$ for all $m+N\in(M/N)\backslash(M-rad(N_{j}/N))$ whenever $i\neq
j$. It is easy to verify that if $\sqrt{(N_{i}:_{R}M)}\subseteq\sqrt
{(N_{j}:_{R}m)}$ for some $m\in M,$ then $\sqrt{(N_{i}/N:_{R}M/N)}%
\subseteq\sqrt{(N_{j}/N:_{R}m+N)}$ for some $m+N\in M/N.$ Also observe that if
$m+N\in(M/N)\backslash(M/N-rad(N_{j}/N))=(M/N)\backslash(M-rad(N_{j})/N)$,
then $m\in M\backslash M-rad(N_{j}).$ Thus, from our assumption $\sqrt
{(N_{i}/N:_{R}M/N)}\nsubseteq\sqrt{(N_{j}/N:_{R}m+N)}$ for all $m+N\in
(M/N)\backslash(M/N-rad(N_{j}/N))$ whenever $i\neq j$, we conclude that
$\sqrt{(N_{i}:_{R}M)}\nsubseteq\sqrt{(N_{j}:_{R}m)}$ for all $m\in M\backslash
M-rad(N_{j})$ whenever $i\neq j.$ From our hypothesis and Theorem \ref{av}, we
have $K\subseteq N_{k}$ for some $1\leq k\leq n.$. Consequently, $K/N\subseteq
N_{k}/N$ for some $1\leq k\leq n$; so we are done.
\end{proof}

In view of Theorem \ref{ef} and Theorem \ref{av}, we conclude 1-absorbing
primary avoidance theorem for rings.

\begin{corollary}
Let $I\subseteq I_{1}\cup I_{2}\cup\cdots\cup I_{n}$ be an efficient covering
of ideals $I_{1},I_{2},...,I_{n}$ of a ring $R$ where $n>2.$ If $\sqrt{I_{i}%
}\nsubseteq\sqrt{(I_{j}:x)}$ for all $x\in R\backslash\sqrt{I_{j}}$ whenever
$i\neq j$, then no $I_{i}$ ($1\leq i\leq n$) is a 1-absorbing primary ideal of
$R$.
\end{corollary}

\begin{corollary}
(1-absorbing Primary Avoidance Theorem for Rings) Let $I,I_{1},I_{2}%
,...,I_{n}$ $(n\geq2)$ be ideals of a ring $R$ such that at most two of
$I_{1},I_{2},...,I_{n}$ are not 1-absorbing primary with $I\subseteq I_{1}\cup
I_{2}\cup\cdots\cup I_{n}.$ If $\sqrt{I_{i}}\nsubseteq\sqrt{(I_{j}:x)}$ for
all $x\in R\backslash\sqrt{I_{j}}$ whenever $i\neq j$, then $I\subseteq I_{k}$
for some $1\leq k\leq n.$
\end{corollary}

\end{document}